\documentclass[11pt, a4paper
]{amsart}

\usepackage{amsfonts,graphics, amsthm, amsmath, amscd, amssymb}
\usepackage[draft=false]{hyperref}
\usepackage[T1]{fontenc}
\usepackage{ae}
\usepackage[all]{xy}
\usepackage{xcolor}
\usepackage{tikz-cd}
\usepackage{tikz}
\usetikzlibrary{patterns}
\usepackage{float}

\renewcommand{\le}{\leqslant}

\numberwithin{equation}{section} 
\newtheorem{theorem}{Theorem}[section] 
\newtheorem{definition}{Definition}[section] 
\newtheorem{prop}{Proposition}[section]

\numberwithin{equation}{section}
\theoremstyle{definition}

\newcommand\Q{\mathbb{Q}}

\renewcommand\rho{\varrho}

\title[Erd\H{o}s-Kac type theorem for the number of scattering geodesics]{Erd\H{o}s-Kac type theorem for the number of scattering geodesics on modular surface}



\author{Sudhir Pujahari}
\address{School of Mathematical Sciences, National Institute of Science Education and Research, Bhubaneswar, An OCC of Homi Bhabha National Institute,  P. O. Jatni,  Khurda 752050, Odisha, India.}
\email{spujahari@niser.ac.in}

\author{Punya Plaban Satpathy}
\address{Visitor, School of Mathematical Sciences, National Institute of Science Education and Research, Bhubaneswar, An OCC of Homi Bhabha National Institute,  P. O. Jatni,  Khurda 752050, Odisha, India.}
\email{psatpathy81@gmail.com}

\newcommand{\Ha}{\mathbb{H}}
\newcommand{\M}{\mathcal{M}}
\newcommand{\Mm}{\mathcal{M}_c}

\newcommand{\R}{\mathbb{R}}

\begin{document}

\begin{abstract}
In 1917, Hardy and Ramanujan showed that if $\omega(n)$ is the number of distinct prime factors of a randomly chosen positive integer $n,$ then the normal order of $\omega(n)$ is $\log \log \, n.$  This led Erd\H{o}s and Kac to prove their celebrated result showing a Gaussian behaviour for $\omega(n).$  In this article we prove an Erd\H{o}s-Kac kind result for the number of scattering geodesics on the modular surface with a common sojourn time.

\end{abstract}

\maketitle

\bigskip


\section{Introduction and main results} \label{results}

In 1917, Hardy and Ramanujan \cite{HardyRamanujan1917} showed that if we randomly choose a natural number $n$, it has ``roughly'' $\log \log n$ many distinct prime factors. Motivated by the result of Hardy and Ramanujan, Erd\H{o}s and Kac \cite{Erdoskac} showed that if $\omega(n)$ denotes the number of distinct prime factors of $n$, then $\omega(n)$ is normally distributed with mean and variance $\log \log n$. 
In other words, they showed that the random variable
$
\frac{\omega(n) - \log \log n}{\sqrt{\log \log n}}
$
has a normal distribution with mean $0$ and variance $1$. This result is famously known in the literature as the Erd\H{o}s-Kac Theorem.
The above studies ignited the development of a new branch of mathematics, known as Probabilistic number theory. Since the proof of the Erd\H{o}s-Kac Theorem, a significant number of similar results have been obtained for different mathematical objects.
In \cite{PMM}, Murty, Murty and Pujahari, proved an all-purpose Erd\H{o}s-Kac theorem.  For a more detailed account, the reader is referred to \cite{PMM}. In this article, we prove an Erd\H{o}s-Kac type theorem for the number of scattering geodesics on the modular surface with a common sojourn time.
To state our main result, we recall the following notions:

\newcommand{\f}{\mathcal{F}}
\subsection{The Modular surface $\M$} For a complex number $z = x+iy \in \mathbb{C}$, we set $\Re(z):=x$ and $\Im(z): = y$. To begin our discussion we introduce the
upper half-plane,
\begin{equation}
    \mathbb{H} = \{ z  \in \mathbb{C} \mid \Im(z)>0 \},
\end{equation} and we assign $\mathbb{H}$ the standard hyperbolic metric \( ds^2 = \frac{dx^2 + dy^2}{y^2} \). It is a standard result that geodesics in \(\mathbb{H}\) are either vertical half-lines perpendicular to the $x$-axis or semicircles with center on the $x$-axis. We consider the action of \(\text{PSL}(2, \mathbb{Z})\) on \(\mathbb{H}\) through linear fractional transformations, where a two by two matrix \( A = \begin{pmatrix} a & b \\ c & d \end{pmatrix} \) with integer entries \( a, b, c, d \in \mathbb{Z} \) satisfying \( ad - bc = 1 \), acts on $\mathbb{H}$ as follows :
\[
z \mapsto \frac{az + b}{cz + d}.
\]
These transformations act as isometries in \(\mathbb{H}\), as demonstrated in \cite[Theorem 3.3.1]{Anderson2005-ox}. We use the following fundamental domain \(\mathcal{F}\)  for the action of \(\text{PSL}(2, \mathbb{Z})\) on \(\mathbb{H}\) defined as:
\[
\mathcal{F} = \{ z \in \mathbb{H} \mid 0 \leq \Re(z) \leq 1, |z| \geq 1, |z - 1| \geq 1 \}.
\]

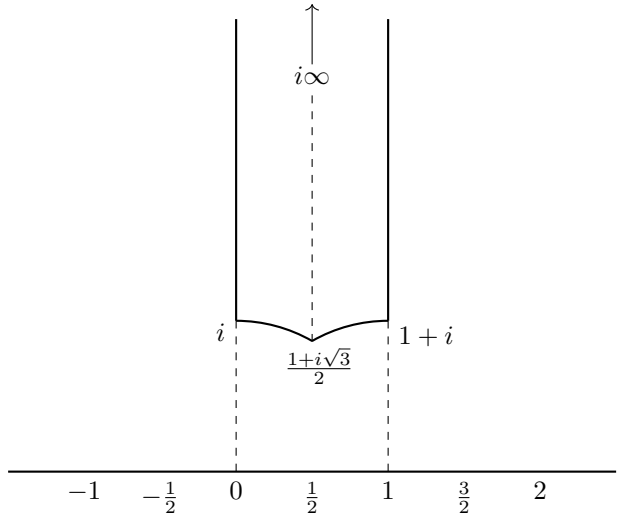
\begin{figure}[htbp]
\centering
\begin{tikzpicture}[scale=2.0, yscale=1.0]  
    \draw[black, thick] (0,1) arc[start angle=90, end angle=60, radius=1];

    \draw[black, thick, shift={(1,0)}] (-0.5,{sqrt(3)/2}) arc (120:90:1);

    \draw[dashed] (0,0) -- (0,1.0); 
    \draw[thick] (1.0,1.0) -- (1.0,3);

    \draw[dashed] (1.0,0) -- (1.0,1.0); 
    \draw[thick] (0,1) -- (0,3);

    \draw[thick] (-1.5,0) -- (2.5,0); 
    
    \node[below] at (-1,0) {\small $-1$};
    \node[below] at (-0.5,0) {\small $-\frac{1}{2}$};
    \node[below] at (0,0) {\small $0$};
    \node[below] at (0.5,0) {\small $\frac{1}{2}$};
    \node[below] at (1,0) {\small $1$};
    \node[below] at (1.5,0) {\small $\frac{3}{2}$};
    \node[below] at (2,0) {\small $2$};
    
    \node[above] at (-0.1,0.8) {\small $i$};
    \node[right] at (0.25,0.7) {\small $\frac{1 + i\sqrt{3}}{2}$};
    \node[right] at (1.0,0.9) {\small $1 + i$}; 
    
     \draw[dashed] (0.5,0.86) -- (0.5,2.5);
    \draw[->] (0.5,2.7) -- (0.5,3.1);
    \node[above] at (0.5,2.5) {\small $i\infty$};

\end{tikzpicture}
\caption{Fundamental domain for the PSL(2,$\mathbb{Z}$) action on $\mathbb{H}$}
\end{figure}

The object in which we are interested is the modular surface defined as the quotient space \(\mathcal{M} = \mathbb{H}/\text{PSL}(2,\mathbb{Z})\) and we assign it with the metric inherited from the upper half-plane. 
It is a standard result that \(\mathcal{M}\) is non-compact; however it has a finite area with respect to the measure induced from the hyperbolic metric on \(\mathbb{H}\). 
Consider the natural projection map \(\pi : \mathbb{H} \longrightarrow \mathcal{M}\). By selecting a large positive real number \(T_0 \gg 0\), we can divide \(\mathcal{M}\) into two disjoint regions, \(\mathcal{M}^{T_0}_c\) and \(\mathcal{M}_{\infty}^{T_0}\), where
\begin{equation*}
    \mathcal{M}_c^{T_0} := \pi({\mathcal{F}} \cap \{\Im(z) \leq T_0\}) \ and  \ \mathcal{M}^{T_0}_{\infty} := \pi(\mathcal{F} \cap \{\Im(z) > T_0\}) .
\end{equation*}
Here, \(\mathcal{M}^{T_0}_c\) is a compact subset of \(\mathcal{M}\), which is often called the \textit{compact core} and \(\mathcal{M}_{\infty}^{T_0}\) is an open neighborhood of the cusp in \(\mathcal{M}\). It is a well-known result that any geodesic in \(\mathcal{M}\) is the projection of a geodesic in \(\mathbb{H}\) under the canonical map \(\pi : \mathbb{H} \longrightarrow \mathcal{M}\). We focus on a family of scattering geodesics in $\M$ that traverse the cusp and return, a concept first examined by Victor Guillemin in \cite{Guillemin1976-xr}.

\subsection{Scattering geodesics in $\M$} We recall the following definition of scattering geodesics from \cite{Guillemin1976-xr};
\begin{definition}
   A geodesic $\gamma:=\gamma(t)$ in $\M$ is called a scattering geodesic if it is contained in $\M_{\infty}^{T_0}$ for both large positive and large negative times $t$.
 
Further, the associated \textit{sojourn time} \( l_{\gamma} \) is defined as the total time elapsed from the time the geodesic $\gamma$ first enters the compact core \( \Mm^{T_0} \) to the time it  leaves \( \Mm^{T_0} \) permanently.

\end{definition}

It follows from the work of Guillemin [Corollary 2,Theorem B1, \cite{Guillemin1976-xr}] that 
for a fixed $T_0 \gg 0$, there are a countable number of scattering geodesics in $\M$.
Recently, we \cite{pujaharipunya} have obtained a Prime number theorem kind result for scattering geodesics. For the reader's convenience, we recall some key details from \cite{pujaharipunya} concerning a correspondence between the set of scattering geodesics $\mathcal{S}$ on the modular surface $\M$ and a certain subset of rationals $\mathcal{G}$ contained in the interval $[0,1)$.

\newcommand{\qq}{\mathcal{G}}

\subsection*{Construction of the Set \( \mathcal{G} \)}

For convenience of the reader, we recall the construction of a subset \( \mathcal{G} \subset [0,1) \) of the rational numbers from \cite[\S 1.4]{pujaharipunya}. We start as follows.

For each integer \( q \geq 2 \), define
\[
I_q := \{ p \in \mathbb{Z}^+ \mid 1 \leq p < q,\ \gcd(p, q) = 1 \},
\]
and let \( \phi(q) = |I_q| \) denote Euler’s totient function. Next, define the set
\begin{equation} \label{sq}
    S_q := \{ p \in \mathbb{Z}^+ \mid 1 \leq p < q,\ p^2 \equiv -1 \pmod{q} \},
\end{equation}
and let \( s_q := |S_q| \) denote its cardinality.

For each \( p \in I_q \), let \( y_{p,q} \in I_q \) be the unique element satisfying \( p y_{p,q} \equiv -1 \pmod{q} \). Define
\[
C_q := \{ p \in I_q \mid p \neq y_{p,q} \}.
\]
The elements of \( C_q \) form unordered pairs \( \boxed{p_1, p_2} \) such that \( p_1 p_2 \equiv -1 \pmod{q} \).

Define the subset
\begin{equation}
    C^*_q := \left\{ \min\{p_1, p_2\} \mid \boxed{p_1, p_2} \text{ is a pair in } C_q \right\},
\end{equation}
which has cardinality
\begin{equation} \label{c*q}
    |C^*_q| = \frac{1}{2}(\phi(q) - s_q).
\end{equation}

Set \( \mathcal{G}_1 := \{0\} \). For \( q \geq 2 \), define
\begin{equation}
    \mathcal{G}_q := \left\{ \frac{p}{q} \mid p \in C^*_q \sqcup S_q \right\}.
\end{equation}
The number of elements in \( \mathcal{G}_q \) is given by
\[
|\mathcal{G}_q| = \frac{1}{2}(\phi(q) + s_q).
\]

Finally, with all the above notations, we define the set \( \mathcal{G} \subset [0,1) \) of rational numbers by the disjoint union
\begin{equation} \label{G}
    \mathcal{G} := \bigsqcup_{q=1}^{\infty} \mathcal{G}_q.
\end{equation}

Next, we recall the following result from \cite{pujaharipunya}, which provides necessary and sufficient conditions for when two vertical geodesics in $\mathbb{H}$ project to the same image in $\M$.

\begin{prop} (cf. \cite[Theorem 1.4]{pujaharipunya}) \label{scateq}
    For $w \in [0,1) \cap \Q$, let $\bar{\gamma}_w$ be the unique geodesic in $\Ha$ joining $w$ to $i\infty$. Given distinct $w_i = \frac{p_i}{q_i} \in  (0,1)\cap \Q$ with $\gcd(p_i,q_i) =1, q_i \geq 2$ for $i =1,2$, the geodesics $\bar{\gamma}_{w_1},\bar{\gamma}_{w_2}$ project onto the same geodesic in $\M$ if and only if the following conditions are satisfied
\begin{enumerate}
    \item $q_1 = q_2= q$.
    \item $q$ divides $p_1p_2+1$.
\end{enumerate}
\end{prop}

Given a scattering geodesic $\gamma \in \M$, it can be lifted to a geodesic $\bar{\gamma}$ in $\mathbb{H}$ connecting a rational point $w$ to $i\infty$. By repeatedly applying the translation map $z \mapsto z+1$, one can choose $w \in \Q \cap [0,1)$. Proposition~\ref{scateq} can then be used to establish a bijection between the set of scattering geodesics $\mathcal{S}$ on the modular surface $\M$ and the set $\qq$ defined in \eqref{G} (see \cite[Theorem~1.5]{pujaharipunya} for details). Then we can immediately deduce the following result from \cite[Theorem~1.5]{pujaharipunya}.
\begin{theorem} \label{scat}
    Let \( \mathcal{S} \) be the set of scattering geodesics in \( \M \). For each integer \( q \geq 1 \), there exists a subset \( \mathcal{S}_q \subset \mathcal{S} \) consisting of 
    $ n_q := \frac{\phi(q) + s_q}{2} $
    distinct scattering geodesics, each with common sojourn time \( 2 \log(qT_0) \). Here, \( \phi \) is Euler's totient function, \( s_1 = 1 \), and for \( q \geq 2 \), \( s_q \) denotes the cardinality of the set
    \[
        \left\{ p \mid 1 \leq p < q,\ p^2 \equiv -1 \pmod{q} \right\}.
    \]
    Moreover, the full set \( \mathcal{S} \) can be written as the disjoint union
    \[
        \mathcal{S} = \bigsqcup_{q=1}^\infty \mathcal{S}_q.
    \]
\end{theorem}

We now state the main result of this paper.

\begin{theorem} \label{main}
    For any $a \in \R$ we have,

\begin{equation*}
    \lim_{x \longrightarrow \infty} \frac{1}{x}\bigg|\bigg\{1\leq q \leq x \mid \omega(n_q) - \frac{1}{2}(\log \log x)^2 \leq \frac{a}{\sqrt{3}}(\log \log x)^{3/2} \bigg\}\bigg| = \Phi(a).
\end{equation*}
 where $\displaystyle \Phi(a): = \frac{1}{\sqrt{2\pi}} \int_{-\infty}^a e^{-\frac{y^2}{2}}dy$ is the cumulative distribution function for the standard normal distribution.
\end{theorem}
\pagebreak

 \begin{figure}[h]
    \centering
    \includegraphics[width=1.0\textwidth]{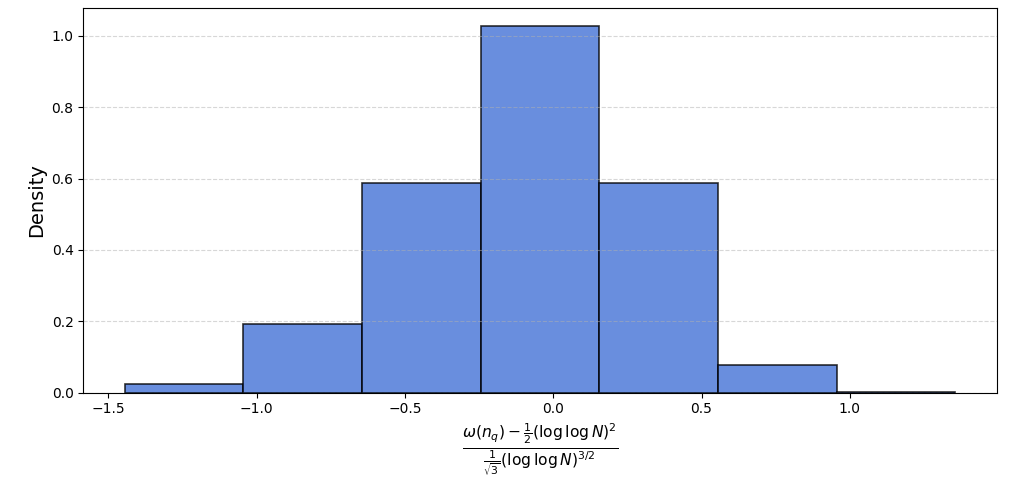}
    \caption{Density Histogram of $\displaystyle \frac{\omega(n_q) -\frac{1}{2}(\log \log N)^2 }{\frac{1}{\sqrt{3}}(\log \log N)^{3/2}}$ with \\ $1 \leq q \leq N=10^7$ .}
    \label{Erdos kac hist}
\end{figure}
The Python script used to generate the plot shown in Figure 2 is publicly available at \href{https://github.com/PunyaPlabanSatpathy/Erdos-kac-plot}{GitHub repository}.

In the next section, we discuss the preliminary results needed to prove our main result.
\section{Preliminaries}

In this section, we discuss the results that we need to prove Theorem \ref{main}. We start with a discussion of the arithmetic function $s_q$ defined in Theorem \ref{scat}. Recall that we have set, $s_1 = 1$ and for $q \geq 2$, $s_q$ denotes the cardinality of the set
\begin{equation*}
    \{p \mid 1\leq p < q, p^2 \equiv -1 \pmod{q}\}.
\end{equation*}

We recall the following result, which was established in \cite[Theorem 2.2]{pujaharipunya};
\newcommand{\kk}{ \mathcal{O}}
\begin{prop}\label{sqdes}
    Let $s_q$ be the arithmetic function defined in Theorem \ref{scat}. Then $s_q \neq 0$ if and only if $q \in \kk$ or $q/2 \in \kk$, where $\kk$\footnote{We also make the convention that $1 \in \kk$.} is the subset of natural numbers such that $n \in \kk \iff $ all the prime factors of $n$ are of the form $4k+1$. 
\end{prop}

For $x \geq 1$, let $A(x)$ be the cardinality of the following set,
\begin{equation*}
    \{1\leq q \leq x \mid s_q \neq 0\}.
\end{equation*}
Our first result involves understanding the asymptotics of $A(x)$ for large $x$. More specifically, we have the following;

\begin{prop}\label{sqes}
    For sufficiently large $x$ we have,

\begin{equation}
    A(x) : = |\{1\leq q \leq x \mid s_q \neq 0\}| \sim \frac{\alpha x}{\sqrt{\log x}},
\end{equation}    
where
\begin{equation*}
    \alpha =\frac{3}{2\pi}\prod_{p \equiv 1 \bmod{4}} (1-p^{-2})^{-1/2}.
\end{equation*}
\end{prop}
\begin{proof}
    Consider the L-function $F(s)$ defined as follows,
 \begin{equation*}
     F(s) = \sum_{q=1}^{\infty}\frac{a_q}{q^s},
 \end{equation*}   
 where $a_q =1$ if $s_q \neq 0$ and $a_q =0$ otherwise. Then
 \begin{equation*}
     A(x) = \sum_{q\leq x}a_q.
 \end{equation*}
  We can rewrite $F(s)$ as follows using proposition \ref{sqdes},

\begin{equation*}
    F(s) = \sum_{q \in \kk} \frac{1}{q^s} +\sum_{q/2 \in \kk } \frac{1}{q^s}  = (1+2^{-s}) \sum_{q \in \kk}\frac{1}{q^s}.
\end{equation*}
Note that $\displaystyle \sum_{q \in \kk}\frac{1}{q^s}$ and therefore $F(s)$ converges absolutely for $\Re(s) >1 $ and consequently a holomorphic function in the region $\Re(s)>1$. It is also clear from heuristic arguments that $F(s)$ has singular behaviour at $s=1$. In order to understand the nature of this singularity at $s=1$, we look at $(F(s))^2$.

\begin{align*}
 (F(s))^2& =(1+2^{-s})^2\prod_{p \equiv 1 \bmod{4}}(1-p^{-s}) ^{-2} \\
 & = (1+2^{-s})^2\prod_{p \equiv 1 \bmod{4}} \frac{1-p^{-2s}}{(1-p^{-s})^2} \bigg(\prod_{p \equiv 1 \bmod{4}} (1-p^{-2s})^{-1}\bigg) \\
 & = (1+2^{-s})^2\prod_{p \equiv 1 \bmod{4}} \frac{1+p^{-s}}{1-p^{-s}}   \bigg(\prod_{p \equiv 1 \bmod{4}} (1-p^{-2s})^{-1}\bigg) \\
 & = (1+2^{-s}) \prod_{p}(1+p^{-s}) \bigg(\prod_{p \equiv 1 \bmod{4}} (1-p^{-s})^{-1}\prod_{p \equiv 3 \bmod{4}} (1+p^{-s})^{-1} \bigg) G(s) \\
 & = (1+2^{-s})L(s,\chi)\prod_{p} \frac{1-p^{-2s}}{1-p^{-s}} G(s) \\
 & = \frac{(1+2^{-s})L(s,\chi)\zeta(s)G(s)}{\zeta(2s)},
\end{align*}
where $\zeta(s)$ is the Riemann zeta function, $L(s,\chi)$ is the Dirichlet's L-function associated to the non-principal character $\chi$ modulo 4 defined by 
 \[
\chi(n) =
\begin{cases}
  0, & \text{if } n \text{ is even}, \\
  1, & \text{if } n \equiv 1 \bmod{4}, \\
  -1, & \text{if } n \equiv 3 \bmod{4},
\end{cases}
\]
and 
\begin{equation*}
    G(s):=\prod_{p \equiv 1 \bmod{4}} (1-p^{-2s})^{-1}.
\end{equation*}
We know that $L(s,\chi)$ is an entire function. $G(s)$ is holomorphic and non-vanishing in the region $\Re(s)> 1/2$.  In addition, $\zeta(s)$ is holomorphic in the region $\Re(s) \geq 1$ except for a simple pole at $s = 1$.  Hence, the function $(F(s))^2$ is holomorphic in the region $\Re(s) \geq 1$ except for a simple pole at $s = 1$ with residue

\begin{equation*}
    C = \frac{9}{4\pi}\prod_{p \equiv 1 \bmod{4}} (1-p^{-2})^{-1}.
\end{equation*}
Then applying \cite[Theorem 2]{IkeharaDelange} we get

\begin{equation}
     A(x) = |\{1\leq q \leq x \mid s_q \neq 0\}| \sim \frac{\alpha x}{\sqrt{\log x}},
\end{equation} 
where
\begin{equation*}
    \alpha =\frac{3}{2\pi}\prod_{p \equiv 1 \bmod{4}} (1-p^{-2})^{-1/2}.
\end{equation*}
\end{proof}
Next, we prove that the proportion of integers \( q \le x \) for which
    $| \omega(n_q) - \omega(\phi(q))| > 1 $  tends to zero as \( x \to \infty \).
\begin{prop} \label{Exset}
    For $q \geq 1$, let $n_q$ be the arithmetic function defined in Theorem \ref{scat} and $\omega(q)$ be the number of distinct prime divisors of $q$. Let $x \gg 0$ and consider the following set $E(x)$ defined by
 \begin{equation}\label{exdef}
  E(x):= \{ 1\leq q \leq x \mid  |\omega(n_q) -\omega(\phi(q))|> 1\}.   
 \end{equation}
 Then $|E(x)| = o(x)$ as $x \to \infty$.
\end{prop}
\begin{proof}
   Note that for $q=1,2$ one can check that $|\omega(n_q) -\omega(\phi(q))|\leq 1$. So, consider $3 \leq q \leq x$, there are two cases to consider, that is, whether $s_q=0$ or $s_q \neq 0$. If $s_q =0$, then $\omega(n_q) = \omega(\phi(q)/2)$ and it can easily be verified that $|\omega(\phi(q)/2) -\omega(\phi(q))|\leq 1$. Otherwise, for $s_q \neq 0$, we have
\begin{equation*}
    \bigg|\bigg\{1\leq q \leq x \mid  |\omega(n_q) -\omega(\phi(q))|> 1 \bigg\}\bigg| \leq  A(x) = |\{1\leq q \leq x \mid s_q \neq 0\}|
\end{equation*}
and the claim follows from Proposition \ref{sqes}.
\end{proof}

Next we state the Erd\H{o}s-Kac theorem for the Euler's $\phi$-function proved by Erd\H{o}s and Pomerance in \cite{Erdos1985-la}. We use this in the next section to prove the main theorem.

\begin{prop}(cf. \cite[Theorem 3.2]{Erdos1985-la}) \label{erdoskacphi}
For any $a \in \R$ we have,

\begin{equation*}
    \lim_{x \longrightarrow \infty} \frac{1}{x}\bigg|\bigg\{1\leq q \leq x \mid \omega(\phi(q)) - \frac{1}{2}(\log \log x)^2 \leq \frac{a}{\sqrt{3}}(\log \log x)^{3/2} \bigg\}\bigg| = \Phi(a).
\end{equation*}
    where $\displaystyle \Phi(a): = \frac{1}{\sqrt{2\pi}} \int_{-\infty}^a e^{-y^2/2}dy$ is the cumulative distribution function for the standard normal distribution.
\end{prop}

 \section{Proofs}

In this section, we prove Theorem \ref{main}. To prove the claim, for sufficiently large $x$ and $a \in \R$, define 
\[
\begin{aligned}
I(q,x,a) &:= \begin{cases}1 & \text{if }\frac{\omega(\phi(q)) - f(x)}{g(x)} \le a,\\0 & \text{otherwise},\end{cases} & and  \ 
I_1(q,x,a) &:= \begin{cases}1 & \text{if }\frac{\omega(n_q) - f(x)}{g(x)} \le a,\\0 & \text{otherwise},\end{cases}
\end{aligned}
\]
where
$f(x):= \tfrac12(\log\log x)^2,$ \text{and} $g(x):=\tfrac1{\sqrt3}(\log\log x)^{3/2}$.

We get the required result by comparing the following sums;

\[
S(x,a) = \sum_{q \le x} I(q,x,a), \ and \ S_1(x,a) = \sum_{q \le x} I_1(q,x,a).
\]

By Proposition \ref{erdoskacphi},

\[
\lim_{x \to \infty} \frac{S(x,a)}x = \Phi(a) \quad\text{for every }a\in\mathbb R.
\]

It therefore suffices to show
\begin{equation}\label{eq:goal}
\lim_{x \to \infty} \frac{S_1(x,a)}x = \Phi(a) ,
\end{equation}
which is equivalent to show
\begin{equation}\label{eq:reduce}
\lim_{x \to \infty} \frac{T(x,a)}x = 0, \ \text{where} \ T(x,a):= S(x,a) - S_1(x,a).
\end{equation}

In order to estimate $T(x,a)$, we split its underlying sum into two parts, one corresponding to $q \in E(x)$ and the other corresponding to $q\notin E(x)$.

\medskip

\textbf{Case (I): Contribution from  $q \in E(x)$.} Recall that
\begin{equation*}
    E(x) = \{1 \le q \le x : |\omega(n_q) - \omega(\phi(q))| > 1\},
\end{equation*}

and by Proposition \ref{Exset} we have $|E(x)| = o(x)$ as $x \to \infty$. Define

\[
T_1(x,a) := \sum_{{q\le x} \atop {q\in E(x)}} \bigl(I(q,x,a) - I_1(q,x,a)\bigr).
\]

Since $|I(q,x,a) - I_1(q,x,a)|\le1$ for every $q,$ we have

\[
|T_1(x,a)| \le |E(x)| = o(x) \mbox{ as } x \longrightarrow \infty.
\]

\medskip

\textbf{Case (II): Contribution from $q \notin E(x)$.} Let

\[
T_2(x,a) = \sum_{{q\le x} \atop {q\notin E(x)}} \bigl(I(q,x,a) - I_1(q,x,a)\bigr).
\]

For $q\notin E(x)$ we have $|\omega(n_q)-\omega(\phi(q))|\le1$, so

\[
\left|\frac{\omega(n_q)-f(x)}{g(x)} - \frac{\omega(\phi(q))-f(x)}{g(x)}\right| \le \frac1{g(x)}.
\]

Hence, for $q \leq x, q\notin E(x)$ we have

\begin{equation} \label{phincomp}
    \frac{\omega(\phi(q)) - f(x)}{g(x)} - \frac{1}{g(x)} \leq \frac{\omega(n_q) - f(x)}{g(x)} \leq \frac{\omega(\phi(q)) - f(x)}{g(x)} + \frac{1}{g(x)}
\end{equation}
Now
\[
T_2(x,a) = \sum_{q\in E_1(x,a)}(I-I_1) + \sum_{q\in E_2(x,a)}(I-I_1),
\]

where
\[
\begin{aligned}
E_1(x,a) &:= \{q\le x:\,q\notin E(x),\,I(q,x,a)=1,\,I_1(q,x,a)=0\},\\
E_2(x,a) &:= \{q\le x:\,q\notin E(x),\,I(q,x,a)=0,\,I_1(q,x,a)=1\}.
\end{aligned}
\]

We now justify that both $E_1(x,a)$ and $E_2(x,a)$ lie within the following set defined by

\[
\left\{ q \le x : a - \frac{1}{g(x)} < \frac{\omega(\phi(q)) - f(x)}{g(x)} \le a + \frac{1}{g(x)} \right\}.
\]

Indeed, suppose $q \in E_1(x,a)$. Then $I(q,x,a) = 1$ and $I_1(q,x,a) = 0$, which implies

\[
\frac{\omega(\phi(q)) - f(x)}{g(x)} \le a \quad \text{and} \quad \frac{\omega(n_q) - f(x)}{g(x)} > a.
\]

These two inequalities along with \eqref{phincomp} tell us that $q \in E_1(x,a)$ implies

\[
a - \frac{1}{g(x)} < \frac{\omega(\phi(q)) - f(x)}{g(x)} \le a.
\]

A similar argument for $q \in E_2(x,a)$ gives

\[
a < \frac{\omega(\phi(q)) - f(x)}{g(x)} \le a + \frac{1}{g(x)}.
\]

Hence, in both cases, we have

\[
a - \frac{1}{g(x)} < \frac{\omega(\phi(q)) - f(x)}{g(x)} \le a + \frac{1}{g(x)},
\]

and thus both $E_1(x,a)$ and $E_2(x,a)$ are contained in

\[
\left\{ q \le x : a - \frac{1}{g(x)} < \frac{\omega(\phi(q)) - f(x)}{g(x)} \le a + \frac{1}{g(x)} \right\}.
\]
 Using Proposition \ref{erdoskacphi}, we have

\[
|E_1(x,a)| + |E_2(x,a)| = o(x) \ \text{as $x \to \infty$},
\]

and thus $T_2(x,a)=o(x) \ \text{as $x \to \infty$}$.

\medskip

Combining case (I) and case (II) yields

\[
T(x,a) = T_1(x,a) + T_2(x,a) = o(x) \ \text{as $x \to \infty$}.
\]

This establishes \eqref{eq:reduce} and completes the proof of Theorem \ref{main}. \qed

\vspace{1.0cm}
{\bf Acknowledgments:} P.P. Satpathy thanks NISER for providing a great working environment. S. Pujahari acknowledges support from the Science and Engineering Research Board [SRG/2023/000930].

\providecommand{\bysame}{\leavevmode\hbox to3em{\hrulefill}\thinspace}
\providecommand{\MR}{\relax\ifhmode\unskip\space\fi MR }
\providecommand{\MRhref}[2]{%
  \href{http://www.ams.org/mathscinet-getitem?mr=#1}{#2}
}
\providecommand{\href}[2]{#2}

\end{document}